\documentclass[11pt]{ip-journal}
\usepackage{graphicx}
\usepackage[mathscr]{euscript}

\newtheorem{theorem}{Theorem}[section]
\newtheorem{corollary}[theorem]{Corollary}
\newtheorem{lemma}[theorem]{Lemma}

\newtheorem{proposition}[theorem]{Proposition}

\newtheorem{definition-proposition}[theorem]{Definition-Proposition}
\newtheorem{lemma-notation}[theorem]{Lemma-Notation}

\theoremstyle{definition}
\newtheorem{definition}[theorem]{Definition}
\newtheorem{modification}[theorem]{Modification}
\newtheorem{example}[theorem]{Example}
\newtheorem{remark}[theorem]{Remark}

\newtheorem{notation}[theorem]{Notation}

\newcommand{\N}{\mathbb{N}}

\newcommand{\Z}{\mathbb{Z}}

\newcommand{\R}{\mathbb{R}}
\newcommand{\C}{\mathbb{C}}

\newcommand{\twopartdef}[4]
{
	\left\{
		\begin{array}{ll}
			#1 & \mbox{if } #2 \\
			#3 & \mbox{if } #4
		\end{array}
	\right.
}

\newcommand{\twobytwo}[4]
{
	\begin{pmatrix}
		#1 & #2 \\
		#3 & #4 \\
	\end{pmatrix}
}

\begin{document}
\title{Looijenga's Conjecture via Integral-Affine Geometry}

\author{Philip Engel}
\thanks{Research partially supported by NSF grant DMS-1502585.}
\address{Harvard University \\ Cambridge, MA 02139 \\ United States \\}
\email{engel@math.harvard.edu}

\begin{abstract} A cusp singularity is a surface singularity whose minimal resolution is a cycle of smooth rational curves meeting transversely. Cusp singularities come in naturally dual pairs. In 1981, Looijenga proved that whenever a cusp singularity is smoothable, the minimal resolution of the dual cusp is an anticanonical divisor of some smooth rational surface. He conjectured the converse. Recent work of Gross, Hacking, and Keel has proven Looijenga's conjecture using methods from mirror symmetry. This paper provides an alternative proof of Looijenga's conjecture based on a combinatorial criterion for smoothability given by Friedman and Miranda in 1983. \end{abstract}

\maketitle

\section{Introduction}

A {\it cusp singularity} $(\overline{V},p)$ is the germ of a minimally elliptic surface singularity such that the exceptional divisor of the minimal resolution $\pi\,:\,V\rightarrow \overline{V}$ is a reduced anticanonical cycle of smooth rational curves meeting transversely: $$\pi^{-1}(p)=D=D_1+\cdots+ D_n\in|-K_V|.$$ The analytic germ of a cusp singularity is uniquely determined by the self-intersections $D_i^2$ of the components of $D$. Cusp singularities come in naturally dual pairs $(\overline{V},p)$ and $(\overline{V}\,\!'\!,p')$, whose exceptional divisors $D$ and $D'$ are called {\it dual cycles}. For every pair of dual cusps, Inoue \cite{Ino77} constructs an associated {\it hyperbolic Inoue surface}---a smooth, non-algebraic, compact complex surface whose only curves are the components of two disjoint cycles $D$ and $D'$. Contracting $D$ and $D'$ produces a surface with two dual cusp singularities and no algebraic curves.

By working out the deformation theory of the contracted hyperbolic Inoue surface, Looijenga \cite{Loo81} proved that if the cusp with cycle $D'$ is smoothable, then there exists an {\it anticanonical pair} $(Y,D)$---a smooth rational surface $Y$ with an anticanonical divisor $D\in |-K_Y|$ whose components have the appropriate self-intersections. Conversely, Looijenga conjectured that the existence of such an anticanonical pair $(Y,D)$ implies the smoothability of the cusp with cycle $D'$. Recently, work of Gross, Hacking, and Keel proved Looijenga's conjecture using methods from mirror symmetry \cite{GHK11}. In this paper, we provide an alternative proof of Looijenga's conjecture.

In the first section, we review foundational material on cusp singularities, hyperbolic Inoue surfaces, anticanonical pairs, and discuss the main result of Friedman-Miranda \cite{FM83}: The cusp with resolution $D'$ is smoothable if there exists a simple normal crossings surface $\mathcal{X}_0=\bigcup V_i$ satisfying certain combinatorial conditions.

We begin the second section by defining the notion of a triangulated integral-affine surface (with singularities). Then, we associate a triangulated integral-affine surface, called the {\it pseudo-fan}, to any anticanonical pair $(V,D)$ and describe two surgeries on the pseudo-fan that correspond to blowing up a point on a component of $D$ and smoothing a node of $D$.  We describe a natural triangulated integral-affine structure on the dual complex $\Gamma(\mathcal{X}_0)$ of a surface $\mathcal{X}_0$ satisfying the conditions of Friedman-Miranda.

In the third section, we define two of surgeries on an integral-affine surface $S$---an internal blow-up and a node smoothing. Both surgeries appear in Symington's work \cite{Sym03} on almost toric fibrations of four-dimensional symplectic manifolds. Assuming certain conditions, there is a unique symplectic four-manifold $Y$ with a Lagrangian torus fibration $$(Y,D,\omega)\rightarrow S$$ which attains certain allowable singular fibers over the singular points of $S$. An internal blow-up or node smoothing of $S$ is the base of a Lagrangian torus fibration of an internal blow-up or node smoothing of $(Y,D,\omega)$, respectively. When $S$ is a disc whose boundary satisfies a negativity condition, we define an integral-affine completion $\hat{S}$ to a sphere by attaching a cone $C$ with a distinguished vertex $v_0$ to the boundary of $S$. Finally, we define the notion of an order $k$ refinement $S[k]$ of an integral-affine surface $S$.

In the fourth section, we construct from an anticanonical pair $(Y,D)$ a surface $\mathcal{X}_0$ satisfying the conditions of the theorem of Friedman-Miranda, thus proving Looijenga's conjecture:

\begin{enumerate}

\item We express $(Y,D)$ as a sequence of internal blow-ups and node smoothings of a toric surface $(\overline{Y},\overline{D})$. By performing the analogous surgeries on a moment polygon $\overline{S}$ for $(\overline{Y},\overline{D})$, we produce the base $S$ of an almost toric fibration of a symplectic surface $(Y,D,\omega)$. We complete to $\hat{S}$ and take an order $k$ refinement $\hat{S}[k]$ which admits a triangulation.

\item We show that a neighborhood of every vertex $v_i$ of the triangulation of $\hat{S}[k]$ with $i\neq 0$ is locally modeled by the pseudo-fan of an anticanonical pair $(V_i,D_i)$. The cone point $v_0$ is locally modeled by the pseudo-fan of the hyperbolic Inoue pair $(V_0,D')$. By gluing the surfaces $V_i$ together, we produce a surface $\mathcal{X}_0$ whose dual complex is $\hat{S}[k]$, as a triangulated integral-affine surface.

\end{enumerate}

\noindent The construction of $\mathcal{X}_0$ may be phrased solely in terms of operations on integral-affine surfaces---no results in symplectic geometry are necessary for the proof, but they provide the primary motivation. Our construction is algorithmic, providing a simple normal crossings resolution of at least one smoothing of any smoothable cusp singularity. We describe, without proof, four modifications and generalizations of the construction and conclude by giving an example of the modified construction in the charge three case $Q(Y,D)=3$. 

Our proof differs from the approach of Gross, Hacking, and Keel. Unlike \cite{GHK11} this work relies on the main theorem of \cite{FM83} to prove smoothability of the dual cusp. This approach has advantages and disadvantages. The delicate convergence arguments of \cite{GHK11} are avoided by using the deformation theory of global simple normal crossings varieties. On the other hand, the techniques of this paper can only see one-parameter smoothings, because smoothings with a higher dimensional base need not have SNC resolutions. The connection between the two proofs is still largely unexplored. In terms of mirror symmetry and the Gross-Siebert program, we use a ``fan" construction which appears to be Legendre dual to the ``polytope" construction of \cite{GHK11}. For us, the fundamental piece of data to prove smoothability of the dual cusp is the existence of $(Y,D)$, as a {\it symplectic} anticanonical pair. 

{\bf Acknowledgements.} I thank my advisor, Robert Friedman, for the time he spent editing this paper, and for his numerous suggestions, explanations, and clarifications. I thank Lucas Culler, who suggested a simplifying modification of the original construction. I'd also like to thank Mark Gross, Paul Hacking, and Sean Keel for their comments. In addition, I thank Eduard Looijenga for our productive conversation. Finally, I would like to thank the referee for their careful reading and excellent suggestions.

\section{Background} Let $(\overline{V},p)$ be the germ of a cusp singularity. The exception divisor of the minimal resolution $$\pi^{-1}(p)=D=D_1+\cdots+ D_n$$ is a cycle of smooth rational curves meeting transversely. We define $\ell(D):=n$ to be the length of the cycle. Whenever $n\geq 3$, we assume $D_i\cdot D_{i\pm 1}=1$, with indices taken mod $n$. If $n=1$, then $D$ is an irreducible, nodal rational curve. If $n=2$, then $D$ is the union of two smooth rational curves that meet transversely at two distinct points. We define $$d_i :=\twopartdef{-D_i^2}{n>1}{2-D_i^2}{n=1.}$$ The analytic germ of a cusp singularity is uniquely determined by the cycle ${\bf d}:=(d_1,\dots,d_n)$, well-defined up to cyclic permutation and orientation. As $D$ is contractible, Artin's contractibility criterion implies that the intersection matrix $[D_i\cdot D_j]$ is negative-definite. No component of $D$ is an exceptional curve, because it is a minimal resolution. The negative-definite condition is then equivalent to \begin{align*} &d_i\geq 2\,\,\,\,\,\,\,\,\,\,\textrm{ for all }i, \\ &d_i\geq 3\,\,\,\,\,\,\,\,\,\,\textrm{ for some }i.\end{align*}

Cusp singularities arise in the classification of complex analytic surfaces. Amongst those of Type $\textrm{VII}_0$ are the hyperbolic Inoue surfaces, which have first Betti number $b_1=1$ and Kodaira dimension $\kappa=-\infty$. For a construction, see \cite{Ino77}. The only curves on a hyperbolic Inoue surface $V$ are the components of two contractible cycles $D$ and $D'$ of rational curves satisfying $D+D'\in|-K_V|$. Both cycles can be blown down to give a surface $(\overline{V},p,p')$ with two {\it dual cusps}. For any cusp singularity $p$, there is a construction of $\overline{V}$ as the compactification of a quotient of $\mathbb{H}\times \C$ by a discrete group action (hence the terminology ``cusp"). Suppose that the cycle of negative self-intersections of $D$ is $${\bf d}=(d_1,\dots,d_n)=(a_1+3,\,\underbrace{2,\dots,2}_{b_1},\,\dots ,\, a_k+3,\,\underbrace{2,\dots,2}_{b_k})$$ with $a_i, b_i\geq 0$. An explicit formula for the negative self-intersections of the dual cycle $D'$ may be given from those of the original cycle $D$ by interchanging $a_i$ and $b_i$: $${\bf d'}=(d_1',\dots,d_s')=(b_1+3,\,\underbrace{2,\dots,2}_{a_1},\,\dots ,\, b_k+3,\,\underbrace{2,\dots,2}_{a_k}).$$ 

Let $(\overline{V},p,p')$ denote the (disconnected) germ of the two cusp singularities on the doubly contracted hyperbolic Inoue surface $\overline{V}$. Looijenga \cite{Loo81} proves:

\begin{theorem}\label{loo1} $\overline{V}$ has a universal deformation which is semi-universal for the germ $(\overline{V},p,p')$. \end{theorem}

Suppose that $p'$ is smoothable. By Theorem \ref{loo1}, there exists a deformation $$\pi:\mathcal{V}\rightarrow \Delta$$ over an analytic disc, with $\mathcal{V}_0=\overline{V}$, which keeps the germ $(\overline{V},p)$ constant while smoothing the germ $(\overline{V},p')$. Any fiber $\mathcal{V}_t$ with $t\neq 0$ is a surface with a single cusp singularity $p=p_t$. Simultaneously resolving the singularities $p_t$ produces a family $\mathcal{Y}\rightarrow \Delta$ whose central fiber is the partially contracted hyperbolic Inoue surface with cusp singularity $p'$ and whose general fiber is a smooth surface. Any fiber $\mathcal{Y}_t$ with $t\neq 0$ is a simply connected surface with anticanonical divisor $D$, which by the classification of surfaces must be rational. Hence, the following corollary to Theorem \ref{loo1}:

\begin{corollary}\label{loo2} Suppose that $D'$ contracts to a smoothable cusp singularity. Then, $D$ is the anticanonical divisor of some rational surface.\end{corollary} Looijenga conjectured the converse, which by the work of Gross, Hacking, and Keel \cite{GHK11} on mirror symmetry for anticanonical pairs, is now a theorem:

\begin{theorem}[Looijenga's Conjecture]\label{looij} If $D$ is the anticanonical divisor of some rational surface, then the cusp singularity associated to $D'$ is smoothable. \end{theorem}

Now we review some basic facts about rational surfaces with an anticanonical cycle $D$. Such surfaces are log generalizations of K3 surfaces: They are simply connected surfaces with a global non-vanishing meromorphic $2$-form with poles along a simple normal crossings divisor $D$:

\begin{definition} An {\it anticanonical pair} or simply {\it pair} $(Y,D)$ is a rational surface $Y$ with an anticanonical divisor $D$ equal to a cycle of rational curves $$D=D_1+\cdots +D_n\in |-K_Y|$$ meeting transversely. A {\it negative-definite pair} satisfies the additional condition that the intersection matrix $[D_i\cdot D_j]$ is negative-definite. A {\it toric pair} is a pair where $Y$ is a toric surface and $D$ is the toric boundary.

\end{definition}

Let $E$ be an exceptional curve on $(Y,D)$---by this we always mean an exceptional curve of the first kind, i.e. $E\cong \mathbb{P}^1$ with $E^2=-1$. Contracting $E$ gives a smooth anticanonical pair: $$\pi:(Y,D)\rightarrow(\overline{Y},\overline{D}).$$ If $E$ is a component of $D$, then $E$ contracts to a node point of the cycle $\overline{D}$. In this case, we call $\pi$ a {\it corner blow-up}. If $E$ is not a component of $D$, then $E$ intersects $D$ at one of its smooth points. Thus, $E$ contracts to a smooth point of the cycle $\overline{D}$, in which case, we call $\pi$ an {\it internal blow-up}. Conversely, given any anticanonical pair, we can blow-up either a corner or a smooth point of the cycle to produce a new anticanonical pair. In addition to blowing up on $D$, we can smooth any node of $D$:

\begin{proposition}\label{def} Let $(V,D)$ be an anticanonical pair and let $p$ be a node of $D$. There exists a family of anticanonical pairs $$(\mathcal{V},\mathcal{D})\rightarrow \Delta$$ over the disc whose central fiber is $(V,D)$ such that $\mathcal{D}\rightarrow \Delta$ is a smoothing of the node $p$.  \end{proposition}

\begin{proof} The proposition follows from Corollary 3.6 of \cite{Fri15}, which proves the result for any subset of the nodes of $D$. Roughly, the deformations of $(V,D)$ surject onto the deformations of $D$. \end{proof}

\begin{definition} The {\it charge} of a cycle $D$ or of a pair $(Y,D)$ is defined by the formula $$Q(D)=Q(Y,D):=12+\sum_{i=1}^n \,(d_i-3)=12+\sum_{i=1}^n (a_i-b_i).$$ \end{definition} \noindent  The formula for the dual cusp $D'$ interchanges $a_i$ with $b_i$ and thus $Q(D)+Q(D')=24$. The charge of an anticanonical pair $(Y,D)$ is essentially a measure of how far it is from being toric: All toric pairs have charge zero, while all other anticanonical pairs have positive charge.

\begin{remark}\label{changes} Let $(Y,D)$ be a pair. An internal blow-up on $D_i$ changes the cycle ${\bf d}$ by $$(\dots,\,d_i,\,\dots)\,\,\mapsto\,\, (\dots,\,d_i+1,\,\dots)$$ and increases the charge by $1$. A corner blow-up at $D_i\cap D_{i+1}$ changes the cycle ${\bf d}$ by $$(\dots,\,d_i,\,d_{i+1},\,\dots)\,\,\mapsto\,\, (\dots,\,d_i+1,\,1,\,d_{i+1}+1,\,\dots)$$ and keeps the charge constant. A node smoothing at $D_i\cap D_{i+1}$ changes the cycle ${\bf d}$ by $$(\dots,\, d_i,\,d_{i+1},\,\dots )\,\,\mapsto\,\, (\dots,\,d_i+d_{i+1}-2,\,\dots)$$ and increases the charge by $1$. 
\end{remark}

Consider a smoothing family $\mathcal{Y}\rightarrow \Delta$ whose central fiber is the partially contracted hyperbolic Inoue surface with cusp singularity $p'$. Using the same methods as Kulikov \cite{Kul77} and Persson and Pinkham \cite{PP81} in their study of degenerations of K3 surfaces, Friedman and Miranda \cite{FM83} prove that after a finite base change and bi-meromorphic modifications on $\mathcal{Y}\rightarrow \Delta$, we can produce a smooth family $\mathcal{X}\rightarrow \Delta$ such that $\mathcal{D}\in|-K_{\mathcal{X}}|$ is a divisor restricting to $D$ on every fiber and the central fiber is a simple normal crossings surface. In analogy with so-called Type III degenerations of K3 surfaces, the central fiber $$\mathcal{X}_0=\bigcup_{i=0}^nV_i$$ of the family $\mathcal{X}\rightarrow \Delta$ satisfies the following conditions:

\begin{enumerate}
\item[i.] $V_0$ is the hyperbolic Inoue surface with cycles $D$ and $D'$. For $i>0$, the normalization $\tilde{V}_i$ of each $V_i$ is a smooth rational surface.
\item[ii.] Let $D_{ij}$ denote an irreducible double curve of $\mathcal{X}_0$ lying on $V_i$ and $V_j$ (if $V_i$ is not normal, we may have $i=j$). Define $D_i$ to be the union of the double curves $D_{ij}$ contained in $V_i$. Let $\tilde{D}_i$ be the inverse image of $D_i$ under the normalization map $\tilde{V}_i\rightarrow V_i$. Then $$(\tilde{V}_i,\tilde{D}_i)$$ is an anticanonical pair for $i>0$ and $D_0=D'$. 
\item[iii.] (Triple Point Formula) Let $D_{ij}$ be a double curve joining the surfaces $V_i$ and $V_j$. Then $$\left(D_{ij}\big{|}_{\tilde{V}_i}\right)^2+\left(D_{ij}\big{|}_{\tilde{V}_j}\right)^2=\twopartdef{-2}{D_{ij} \textrm{ is smooth}}{0}{D_{ij} \textrm{ is nodal.}}$$
\item[iv.] The dual complex of $\mathcal{X}_0$ is a triangulation of the sphere.
\end{enumerate} 

\begin{definition} We call a surface $\mathcal{X}_0$ satisfying conditions i.-iv. a {\it Type III anticanonical pair} $(\mathcal{X}_0,D)$. \end{definition}

Conditions i.-iv. are the only combinatorial conditions necessary to ensure that $\mathcal{X}_0$ smooths to an anticanonical pair $(Y,D)$ in a family $\mathcal{X}\rightarrow \Delta$. The remaining condition, $d$-semistability, is analytic: $$T^1_{\mathcal{X}_0}:=\mathcal{E}\!{\it xt}^1_{\mathcal{O}_{\mathcal{X}_0}}(\Omega^1_{\mathcal{X}_0},\mathcal{O}_{\mathcal{X}_0})\cong \mathcal O_{sing(\mathcal{X}_0)}.$$ Any Type III anticanonical pair has a topologically trivial deformation to one which is $d$-semistable by \cite{FM83}, Lemma 2.6. Motivated by the result of \cite{Fri83} in the case of Type III K3 surfaces, \cite{FM83} prove that a $d$-semistable Type III anticanonical pair $(\mathcal{X}_0,D)$ smooths to an anticanonical pair $(Y,D)$. By a result of Shepherd-Barron \cite{SB83}, the union of the surfaces $V_i$ for $i>0$ can be contracted to a point, assuming we also contract the cycle $D'$ on $V_0$. Thus, the existence of $(\mathcal{X}_0,D)$ implies that $D'$ is smoothable. Hence, the Friedman-Miranda criterion \cite{FM83}:

\begin{theorem}\label{fm} The cusp singularity associated to $D'$ is smoothable if and only if there exists a Type III anticanonical pair $(\mathcal{X}_0,D)$. \end{theorem}

\begin{notation} To simplify the notation, we will henceforth suppress the tildes on $(\tilde{V}_i,\tilde{D}_i)$ so that $(V_i,D_i)$ denotes a smooth anticanonical pair. In addition, we introduce the convention $$D_{ij}=D_{ij}\big{|}_{V_i}\,\,\textrm{ and }\,\,D_{ji}=D_{ij}\big{|}_{V_j}$$ so that $D_{ij}$ always denotes a curve on the smooth surface $V_i$. Then $D_{ij}$ and $D_{ji}$ have equal image in $\mathcal{X}_0$ but may not be isomorphic. In fact, the image of $D_{ij}$ in $\mathcal{X}_0$ is nodal if and only if exactly one of $D_{ij}$ or $D_{ji}$ is nodal. We define $$d_{ij}:=\twopartdef{-D_{ij}^2}{\ell(D_i)\geq 2}{2-D_{ij}^2}{\ell(D_i)=1.}$$ Then, the triple point formula states that $d_{ij}+d_{ji}=2$ in all cases. \end{notation}

\begin{proposition}[Conservation of Charge] \label{charge} Let $(\mathcal{X}_0,D)$ be a Type III anticanonical pair. Then, $$\sum Q(V_i,D_i)=24.$$ \end{proposition}

\begin{proof} See \cite{FM83}, Proposition 3.7. \end{proof} 

Conservation of charge is analogous to the Gauss-Bonnet formula, where curvature and charge are equated: The sum of the charges is a constant multiple of the Euler characteristic. Toric surfaces, which have charge zero, are ``flat" in some sense. This analogy will take a precise form in the following section; we show that the dual complex $\Gamma(\mathcal{X}_0)$ of a Type III anticanonical pair $\mathcal{X}_0$ admits a natural integral-affine structure with singularities at the vertices corresponding anticanonical pairs $(V_i,D_i)$ of positive charge. Imposing a singular, integral-affine ``fan" structure on the dual complex of a maximally unipotent degeneration of a Calabi-Yau manifold plays a role in the Gross-Siebert program \cite{GS03} for proving the SYZ conjecture. For instance, the case of a Type III degeneration of K3 surfaces is specifically discussed in \cite{GHK11}.

\section{Integral-Affine Surfaces and Type III Anticanonical Pairs} 

Before defining the integral-affine structure on $\Gamma(\mathcal{X}_0)$, we give some general definitions and propositions regarding integral-affine surfaces. A {\it basis triangle} of $\R^2$ is a triangle of area $\frac{1}{2}$ with integral vertices. The edges of a basis triangle pairwise form a lattice basis. 

\begin{definition} A {\it triangulated integral-affine surface with singularities} is a triangulated real surface $S$, possibly with boundary, such that (1) the complement of the vertices $\{v_i\}\subset S$ of the triangulation admits an atlas of charts into $\R^2$ with transition functions valued in the integral-affine transformation group $SL_2(\Z)\ltimes \Z^2$ and (2) the interior of every triangle admits a chart to a basis triangle. \end{definition}

Note that all labeled basis triangles are equivalent, up to a unique integral-affine transformation. An integral-affine surface with singularities has a canonical orientation induced from the standard orientation on $\R^2$. Let $e_{ij}$ denote a directed edge of the triangulation of $S$ going from $v_i$ to $v_j$ and let $f_{ijk}$ denote a triangle whose counterclockwise-ordered vertices are $v_i$, $v_j$, and $v_k$. Within an integral-affine chart containing the interior of the edge $e_{ij}$, we may view $e_{ij}$ as the vector $v_j-v_i$.

\begin{remark} Because $S$ is triangulated into basis triangles, the boundary $P$ is polygonal: There is a decomposition of the boundary $\partial S=P_1+\dots+P_n$ such that each $P_i$ is integral-affine equivalent to a line segment between two lattice points. By convention, we assume that the boundary components $P_i$ are maximal: The union of two distinct boundary components is never integral-affine equivalent to a single line segment between two lattice points. \end{remark}

If the atlas of integral-affine charts on $S\,\!-\{v_i\}$ extends to all vertices $v_i$, then we say that $S$ is {\it non-singular}. For the remainder of the paper, we will implicitly assume that an integral-affine surface has singularities and we will specify when an integral-affine surface is non-singular.

\begin{definition} Let $S_{sing}$ denote the vertices of the triangulation of $S$ to which the integral-affine structure fails to extend.\end{definition}

\begin{remark}\label{mani} Let $S$ be an integral-affine surface with singularities. A small contractible open subset $U\subset S-S_{sing}$ has a chart $\phi_U\,:U \rightarrow\R^2$ which is uniquely defined up to integral-affine transformation. We can then construct a {\it developing map} $$\phi\,:\,S\widetilde{-S}_{sing}\rightarrow \R^2$$ from the universal cover of $S-S_{sing}$ to $\R^2$ by gluing together local charts $\phi_U$ and $\phi_V$ that agree on $U\cap V$. The map $\phi$ is uniquely determined up to post-composition with an element of $SL_2(\Z)\ltimes \Z^2$. The developing map is equivalent to the data of the monodromy representation $$M\,:\,\pi_1(S-S_{sing})\rightarrow SL_2(\Z)\ltimes \Z^2$$ constructed from the parallel transport of the integral-affine structure along a loop. We also make use of the less refined monodromy map $N\,:\,\pi_1(S-S_{sing})\rightarrow SL_2(\Z)$ which projects onto the $SL_2(\Z)$ part of the monodromy. \end{remark}

\begin{definition}\label{si} Let $e_{ik}$ be a directed edge in the interior of a triangulated integral-affine surface. Let $e_{ij}$ and $e_{i\ell}$ be the edges emanating from $v_i$ directly clockwise and counterclockwise to $e_{ik}$. We define the {\it negative self-intersection} $d_{ik}$ of the edge $e_{ik}$ by the formula $$d_{ik}e_{ik}=e_{ij}+e_{i\ell},$$ where we view the edges as lattice vectors in some chart. Note that $d_{ik}$ is an integer because $(e_{ij},e_{ik})$ and $(e_{ik},e_{i\ell})$ are both oriented lattice bases and $d_{ik}$ is independent of the choice of integral-affine chart.  \end{definition}

\begin{proposition}\label{tpf} Let $S$ be a triangulated integral-affine surface with singularities. The formula $d_{ik}+d_{ki}=2$ holds for all interior edges $e_{ik}$.\end{proposition}

\begin{proof} By working within an integral-affine chart, the formula reduces to the following equivalence: $$d_{ik}(v_k-v_i)= (v_j-v_i)+(v_{\ell}-v_i)\!\!\iff \!\!(2-d_{ik})(v_i-v_k)= (v_j-v_k)+(v_{\ell}-v_k).$$ \end{proof}

We now show that the negative self-intersections of the edges uniquely determine the integral-affine structure on $S$:

\begin{proposition}\label{wd} A triangulated integral-affine surface $S$ is uniquely determined by the data of a collection of negative self-intersections $d_{ik}$ for each directed interior edge $e_{ik}$ such that $d_{ik}+d_{ki}=2$. \end{proposition}

\begin{proof} We construct a unique integral-affine surface $S$ from the collection of integers $d_{ik}$. We must declare each triangle $f_{ijk}$ of $S$ to be integral-affine equivalent to a basis triangle. We now show that the integral-affine structure extends to the interior of an edge in a unique manner such that the negative self-intersection of $e_{ik}$ is $d_{ik}$. Given a chart for $f_{ijk}$ there is a unique way to glue to it a basis triangle $f_{ik\ell}$ sharing the edge $e_{ik}$ which satisfies $d_{ik}e_{ik}=e_{ij}+e_{i\ell}$---this equation specifies $e_{i\ell}=d_{ik}e_{ik}-e_{ij}$. See Figure \ref{fig1}, for example. Note that $e_{ik}$ and $e_{i\ell}$ form an oriented lattice basis, so the triangle $f_{ik\ell}$ formed by these vectors is a basis triangle. By Proposition \ref{tpf}, gluing the triangles $f_{k\ell i}$ and $f_{kij}$ along the directed edge $e_{ki}$ using the integer $d_{ki}$ results in the same integral-affine structure on the quadrilateral $f_{ijk}\cup f_{ik\ell}$. Thus, we have defined an integral-affine structure on $S-\{v_i\}$. \end{proof}

\begin{figure}
\begin{centering}
\includegraphics[width = 2.75in]{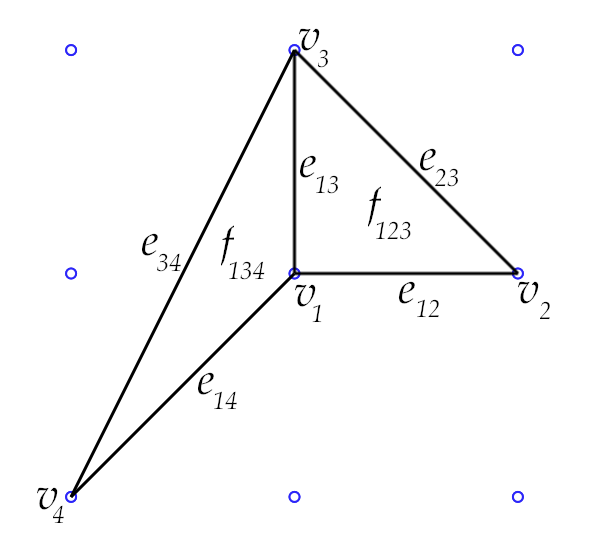} \caption{The integral-affine structure on the union of $f_{123}$ and $f_{134}$ if $d_{13}=-1$.} \label{fig1}
\end{centering}
\end{figure}

We now determine when the non-singular integral-affine structure on $S-\{v_i\}$ extends to an interior vertex $v_i$. But first, we require the following definition:

\begin{definition}\label{pseudofan} Let $(V,D)$ be an anticanonical pair with cycle components $D=D_1+\dots+D_n.$ The {\it pseudo-fan} of $(V,D)$ is a triangulated integral-affine surface whose underlying triangulated surface is the cone over the dual complex of $D$. By Proposition \ref{wd}, it suffices to declare that the negative-self intersection of the directed edge $e_i$ pointing from the cone point to the vertex corresponding to a component $D_i$ is $$d_i =\twopartdef{-D_i^2}{n>1}{2-D_i^2}{n=1}.$$
The imposed integral-affine structure has at most one singularity, at the cone point. Compare to Section 0.3.1 and Lemma 1.3 of \cite{GHK11}. \end{definition}


\begin{proposition}\label{ext} The integral-affine structure on the pseudo-fan of $(V,D)$ extends to the cone point if and only if $(V,D)$ is toric. \end{proposition}

\begin{proof} First suppose that $(V,D)$ is a smooth, complete toric surface. Then, the one-dimensional rays of the fan $\mathfrak{F}$ are spanned by primitive integral vectors $e_i$ such that $(e_i,e_{i+1})$ is an oriented lattice basis. By Section 2.5 of \cite{Ful93}, we have the equation $$-D_i^2e_i=e_{i-1}+e_{i+1}.$$ Thus, the polygon whose vertices are the endpoints of the vectors $e_i$ forms a single chart for the pseudo-fan of $(V,D)$. This integral-affine structure visibly extends over the origin. See Figure \ref{fig2}.

Conversely, if the integral-affine structure on the pseudo-fan of $(V,D)$ extends to the cone point, then we may take a chart around the cone point centered at the origin. The directed edges emanating from the cone point generate the one-dimensional rays of the fan of some toric surface. This toric surface is necessarily isomorphic to $(V,D)$ because the integers $-D_i^2$ determine the fan of $(V,D)$ up to $SL_2(\Z)$ equivalence. Note that $(V,D)$ is automatically toric because $Q(V,D)=0$.\end{proof}

\begin{figure}

\includegraphics[width = 2.75in]{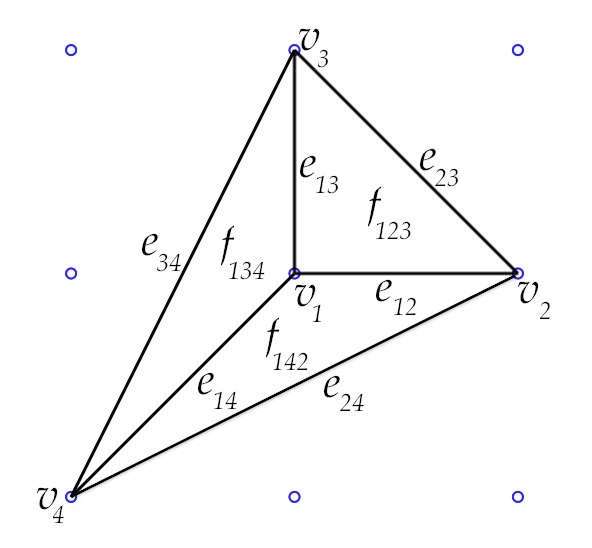} \caption{The integral-affine structure on the pseudo-fan of the toric pair $(\mathbb{P}^2,\triangle)$ where $\triangle$ is a union of three lines in $\mathbb{P}^2$ forming a triangle.}\label{fig2}

\end{figure}

Consider a Type III anticanonical pair $\mathcal{X}_0$. The dual complex $\Gamma(\mathcal{X}_0)$ is a triangulation of $S^2$ whose vertices $v_i$ correspond to components $V_i$, whose directed edges $e_{ij}$ correspond to double curves $D_{ij}$, and whose triangular faces $f_{ijk}$ correspond to triple points $T_{ijk}$. 

\begin{proposition} Let $\mathcal{X}_0$ be a Type III anticanonical pair. The dual complex $\Gamma(\mathcal{X}_0)$ has a triangulated integral-affine structure such that the edge $e_{ij}$ has negative self-intersection $$d_{ij}:=\twopartdef{-D_{ij}^2}{\ell(D_i)\geq 2}{2-D_{ij}^2}{\ell(D_i)=1.}$$ Furthermore, this integral-affine structure extends maximally to $$\Gamma(\mathcal{X}_0)- \{v_i\,:\,Q(V_i,D_i)>0\textrm{ or }\,i=0\}.$$ \end{proposition}

\begin{proof} The triple point formula states that $d_{ij}+d_{ji}=2$, and so by Proposition \ref{wd}, we have the first statement. Proposition \ref{ext} plus the fact that $Q(V_i,D_i)=0$ if and only if a pair $(V_i,D_i)$ is toric imply that the integral-affine structure fails to extend to $v_i$ with $Q(V_i,D_i)>0$. While it is plausible that $Q(V_0,D_0)=0$, the integral-affine structure does not extend to the vertex $v_0$ corresponding to $V_0$---if it did, ${\bf d'}$ would be the negative self-intersection sequence of the boundary components of some toric surface. But $D_0=D'$ is negative-definite, whereas on a toric surface, the boundary components span the Picard group, which has indefinite intersection form. \end{proof}

\begin{definition} Let $v$ be a vertex of a triangulated integral-affine surface. Denote by $star(v)$ the union of the triangles containing $v$. \end{definition}

It is automatic from the definitions that for $\Gamma(\mathcal{X}_0)$ we have an equality between the pseudo-fan of $(V_i,D_i)$ and $star(v_i)$ for all $i$. For convenience, we extend Definition \ref{pseudofan} to allow the pair $(V_0,D_0)$, so that we may also discuss the pseudo-fan of the hyperbolic Inoue surface.

We now record the effect of node smoothings and internal blow-ups on the pseudo-fan:

\begin{proposition}\label{node} Let $(\tilde{V},\tilde{D})$ be a deformation of an anticanonical pair $(V,D)$ such that $\tilde{D}$ is the smoothing of the node $D_{i-1}\cap D_i$. The pseudo-fan of $(\tilde{V},\tilde{D})$ is the result of the following surgery on the pseudo-fan of $(V,D)$: Collapse the triangular face with edges $e_i$ and $e_{i+1}$ to a single edge with negative self-intersection $d_i+d_{i+1}-2$. \end{proposition}

\begin{proposition}\label{blow} Let $(\tilde{V},\tilde{D})$ be an internal blow-up of an anticanonical pair $(V,D)$ on the component $D_i$. The pseudo-fan of $(\tilde{V},\tilde{D})$ is the result of the following surgery on the pseudo-fan of $(V,D)$: Keep the integral-affine structure fixed away from the edge $e_i$ and increase the negative self-intersection of $e_i$ from $d_i$ to $d_{i+1}$. \end{proposition}

\begin{proof} Both Propositions \ref{node} and \ref{blow} follow immediately from Remark \ref{changes}. \end{proof}

If one begins with a toric surface $(V,D)$, the $SL_2(\Z)$ component of monodromy around the cone point after either surgery is conjugate to $$\twobytwo{1}{1}{0}{1}.$$ In particular, monodromy of the pseudo-fan of an internal blow-up of a toric surface is a shear along the edge $e_i$ corresponding to the component receiving the blow-up.

\section{Surgeries on Integral-Affine Surfaces}

We next describe surgeries on integral-affine surfaces that prove useful in the next section when constructing a Type III anticanonical pair $\mathcal{X}_0$. These surgeries are motivated by work of Symington \cite{Sym03} on almost toric fibrations, with further details in Remark \ref{rem}. For the remainder of this section, let $S$ denote a singular integral-affine surface which is homeomorphic to a disc, so that the polygonal boundary $\partial S$ is a circle. That is, the boundary $$\partial S=P_1+\dots+P_n$$ is the union of a sequence of segments $P_i$ put end-to-end, with each segment integral-affine equivalent to a straight line segment between two lattice points. We index the boundary components $P_i$ such that they go counterclockwise around $S$ as $i$ increases. Let $$v_{i,i+1}:=P_i\cap P_{i+1}$$ denote a vertex of $\partial S$ and let $x_i$ and $y_i$ denote the primitive integral vectors emanating from $v_{i,i+1}$ along $P_{i+1}$ and $P_i$, respectively. Thus, $y_{i+1}=-x_i$ in a local chart on $S$ containing the edge $P_i$. We further assume that $(x_i,y_i)$ is an oriented lattice basis. Consequently, the interior angles at the vertices of $P$ are less than $\pi$ in any integral-affine chart.

\begin{remark}\label{rem} Let $(X,\omega)$ by a symplectic, toric surface---that is, a compact symplectic $4$-manifold equipped with a Hamilton two-torus action. If we think of $X$ as a complex toric surface, we can view the Hamiltonian two-torus action as the action of $S^1\times S^1\subset \C^*\times \C^*$. Recall that there is a moment map $$\mu\,:\, (X,\omega)\rightarrow S$$ to a convex planar polygon $S$ (including its interior) such that the toric boundary components of $X$ map to the components of $\partial S$. The general fiber of $\mu$ is a Lagrangian torus, which degenerates on the edges of $S$ to a circle and on the vertices of $S$ to a point. When $[\omega]\in H^2(Y,\Z)$ is integral, the moment polygon can be taken to have integral vertices. Then $S$ satisfies the assumptions of this section---in particular, the vectors $x_i$ and $y_i$ form an oriented lattice basis.

Following Symington \cite{Sym03}, an {\it almost toric fibration} is a Lagrangian fibration $\mu\,:\, (X,\omega)\rightarrow S$ whose general fiber is a smooth $2$-torus, which undergoes symplectic reduction over the boundary $\partial S$, but whose interior fibers may also degenerate to necklaces of spheres at some finite set of points. The {\it almost toric base} $S$ is a generalization of the moment polygon, and has a natural integral affine-linear structure, with $v\in S_{sing}$ whenever the fiber $\mu^{-1}(v)$ is singular. The inverse image of $\partial S$ is an anticanonical divisor of $X$ in the sense of symplectic geometry.

If $S$ is homeomorphic to a disc, then an almost toric fibration over $S$ is a symplectic anticanonical pair $$(Y,D,\omega)\rightarrow S$$ which sends the components of $D$ to the components of $\partial S$. Symington defines two surgeries on $S$: An internal blow-up and a node smoothing (in the terminology of \cite{Sym03}, an almost toric blow-up and a nodal trade, respectively). An internal blow-up of $S$ on the boundary component $P_i$ is the base of an almost toric fibration of an internal blow-up of $(Y,D,\omega)$ on the component $D_i$, and an analogous statement holds for a node smoothing.


\end{remark}

\begin{definition}\label{selfie} Choose a chart of $S$ containing a neighborhood of the edge $P_i$. We define the {\it negative self-intersection} $d_i$ of the boundary component $P_i$ by the formula \begin{align*} d_iy_i&=y_{i-1}-x_i \\ &(=y_{i-1}+y_{i+1}).\end{align*} \end{definition}

If $(Y,D,\omega)\rightarrow S$ is an almost toric fibration, then the negative self-intersection of the edge $P_i$ is equal to the negative self-intersection $-D_i^2$ of the component of $D$ fibering over $P_i$.

We do not define $d_i$ by the formula $d_iy_i=y_{i-1}+y_{i+1}$ because $y_{i+1}$ is not a vector based at a point on $P_i$ and thus, may not be defined in our chosen neighborhood of $P_i$. By extending our chart to a neighborhood of $P_i\cup P_{i+1}$, this definition would become valid. Note that Definition \ref{selfie} fails when $P$ has only one boundary component, as no neighborhood of the single edge is contractible. This problem is resolved by working in a chart on the universal cover of a neighborhood of the boundary component. 

\begin{definition}\label{inter} We define an {\it internal blow-up} of $S$ on the boundary component $P_i$. First, delete a triangle $T\subset S$ satisfying: 
\begin{enumerate}
\item One edge of $T$ is a proper subset of $P_i$. 
\item The remainder of $T$ lies in the interior of $S-S_{sing}$.  
\item $T$ is an integer multiple $n$ of a basis triangle.
\end{enumerate}

\noindent Let $v$ be the unique vertex of $T$ contained in the interior of $S$. Denote by $(e_1,e_2)$ the oriented lattice basis emanating from $v$ along the edges of $T$. See Figure \ref{fig3}. Glue the edge along $e_2$ of $S\,\!- T$ to the edge along $e_1$ of $S\,\!- T$ via the unique affine-linear map which fixes $v$, maps $e_2\mapsto e_1$, and preserves the line containing $P_i$. The resulting integral-affine surface is an {\it internal blow-up} of $S$ on the boundary component $P_i$. Its singular set is $S_{sing}\cup \{v\}$. Call $n$ the {\it size} of the surgery.
\end{definition}

The gluing map is a shear transformation along the line through $v$ parallel to $e_2-e_1$. An internal blow-up does not change the number of boundary components because the lefthand and righthand pieces of $P_i$ are glued into a single line segment. After the surgery, $(x_j,y_j)$ is still an oriented lattice basis for all $j$ because the internal blow-up does not alter the integral-affine structure in the neighborhood of a vertex. 

\begin{figure}
\centering

\includegraphics[width = 2.4in]{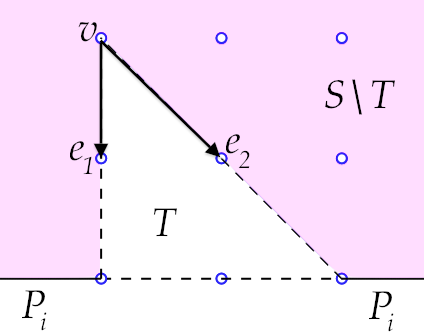}\caption{An internal blow-up on $P_i$ of size $2$.}%
\label{fig3}
\end{figure}

\begin{figure}
\includegraphics[width = 2.2in]{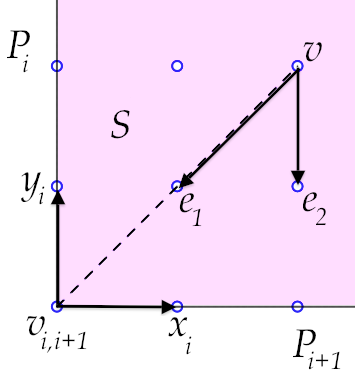}\caption{A node smoothing of $v_{i,i+1}$ of size $2$.}%
\label{fig4}
\end{figure}

\begin{definition}\label{smoo} We define a {\it node smoothing} of $S$ at $P_i\cap P_{i+1}$. For some $n\in \N$, make a cut along the segment from $v_{i,i+1}$ to a point $$v:=v_{i,i+1}+n(x_i+y_i)$$ lying in $S - \partial S$. See Figure \ref{fig4}. Glue the clockwise edge of the cut (from the perspective of $v$) to the counterclockwise edge of the cut by the shearing map which point-wise fixes the line containing the cut and maps $x_i$ to $-y_i$. The resulting integral-affine surface is a {\it node smoothing} of $S$ at $P_i\cap P_{i+1}$. Its singular set is $A\cup \{v\}$. Call $n$ the {\it size} of the surgery.

\end{definition}

Note that even though the gluing fixes the cut point-wise, it alters the integral-affine structure along the cut. Let $e_1:=-x_i-y_i$ be the primitive integral vector emanating from $v$ along the cut, and let $e_2$ be any vector such that $(e_1,e_2)$ is an oriented lattice basis. Then, in the $(e_1,e_2)$ basis, the gluing map is $$\begin{pmatrix} 1 & 1 \\ 0 & 1 \end{pmatrix}.$$ The gluing is independent of the choice of $e_2$ because it is a shear fixing $e_1$. The boundary of a node smoothing of $S$ has one fewer edge than $S$: After a node smoothing, the edges $P_i$ and $P_{i+1}$ are straightened into a single edge because the image of $x_i$ under the gluing map is $-y_i$. As in the case of the internal blow-up, $(x_j,y_j)$ is still an oriented lattice basis for all $j\neq i$ after the surgery because the node smoothing does not alter the integral-affine structure in the neighborhood of a vertex $v_{j,j+1}$ such that $j\neq i$. The vertex $v_{i,i+1}$ ceases to exist after the surgery. 

\begin{proposition}\label{same} An internal blow-up of $S$ on the boundary component $P_i$ transforms the negative self-intersections of the boundary components as follows: $$(\dots, d_i,\dots)\mapsto (\dots,d_i+1,\dots)$$ while smoothing the node $P_i\cap P_{i+1}$ of $S$ transforms the negative self-intersections of the boundary components as follows: $$(\dots,d_i,d_{i+1},\dots)\mapsto (\dots,d_i+d_{i+1}-2,\dots).$$ \end{proposition}

\begin{proof} We omit the proof of the proposition. It follows from straightforward computations involving the gluing matrices and the vectors $x_i$ and $y_i$. Also, it is implicitly proven in \cite{Sym03}, because the negative self-intersection of $P_i$ from Definition \ref{selfie} is equal to the negative self-intersection $d_i$ of the component $D_i$ which maps to $P_i$ under an almost toric fibration.\end{proof}

We now describe how to complete the disc $S$ to a sphere when the boundary is negative, in the appropriate sense:

\begin{proposition}\label{complete} Let $S$ be an integral-affine disc such that adjacent edges of $\partial S$ meet to form lattice bases and the negative self-intersections of components $P_i\subset \partial S$ satisfy \begin{align*} &d_i\geq 2\,\,\,\,\,\,\,\,\,\,\textrm{ for all }i, \\ &d_i\geq 3\,\,\,\,\,\,\,\,\,\,\textrm{ for some }i.\end{align*} Then there is an natural embedding $S\hookrightarrow \hat{S}$ into an integral-affine sphere such that $\hat{S}_{sing}=S_{sing}\cup \{v_0\}$ for a distinguished point $v_0\in \hat{S}-S$.  \end{proposition} 

\begin{proof} Let $U\supset \partial S$ be a collar neighborhood of the boundary of $S$ containing no singular points. Note that $\pi_1(U)=\Z$. Consider the developing map $$\phi\,:\,\tilde{U}\rightarrow \R^2$$ from the universal cover of $U$ to $\R^2$. The universal cover $\widetilde{\partial S}$ of the boundary maps to an infinite lattice polygon in $\R^2$. Each edge $P_i$ is integral-affine equivalent to some interval $[0,m_i]$ on the $x$-axis, for a unique $m_i\in\N$. Then $\phi(\widetilde{\partial S})$ is an infinite sequence of vectors $\{m_iz_i\}_{i\in\Z}$ put end-to-end, such that $(z_{i+1},-z_i)$ is an oriented lattice basis, and $$d_iz_i=z_{i-1}+z_{i+1}$$ for all $i$ (the indices of $m_i$ and $d_i$ are taken mod $n$). We call $\phi(\widetilde{\partial S})$ a {\it discrete hyperbola}. The interior angles of the discrete hyperbola are less than $\pi$ because $(z_{i+1},-z_i)$ is a lattice basis for all $i\in \Z$. One possible image of $\tilde{U}$ under the developing map is shown in Figure \ref{fig7} with the lower edge forming the discrete hyperbola.

\begin{figure}

\includegraphics[width = 2.9in]{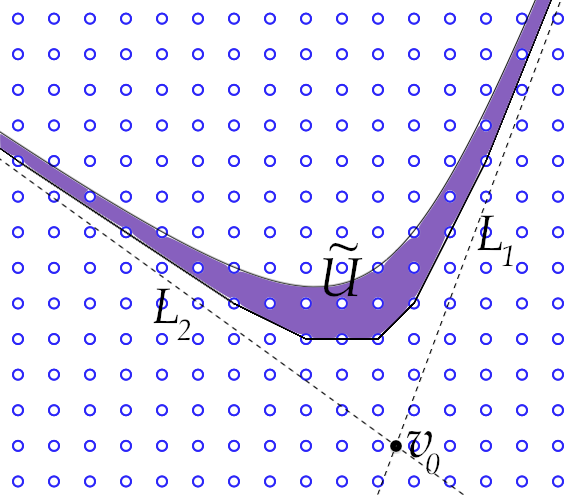}\caption{The image of the developing map of a collar neighborhood of the boundary of $S$.}\label{fig7}

\end{figure}

We claim that the discrete hyperbola has two asymptotic lines $L_1$ and $L_2$ which are the invariant lines of the monodromy transformation $M:=M(\gamma)$ associated to a counterclockwise loop $\gamma$ around the boundary of $S$. The change-of-basis from $(z_i,-z_{i-1})$ to $(z_{i+1},-z_i)$ is $$\begin{pmatrix} d_i & -1 \\ 1 & 0\end{pmatrix}.$$ Therefore, the $SL_2(\Z)$ part of the monodromy $N:=N(\gamma)$ is conjugate in $SL_2(\Z)$ to the product $$\prod_{i=1}^n \begin{pmatrix} d_i & -1 \\ 1 & 0\end{pmatrix}$$ because the counterclockwise monodromy is conjugate to the change-of-basis from $(z_1,-z_0)$ to $(z_{n+1},-z_n)$. By choosing a basis properly, we may assume that $N$ is equal to the above product. The full $SL_2(\Z)\ltimes \Z^2$ monodromy transformation is then $$M\cdot v=N\cdot v+B$$ for some $B\in\Z^2$. Since $(d_1,\dots,d_n)$ is negative-definite, tr$\,N>2$ and therefore $N$ has two distinct, irrational positive eigenvalues. We solve the equation $v=N\cdot v +B$ to find the unique, rational fixed point $$v_0:=(I-N)^{-1}B$$ of $M$. Then $L_1$ and $L_2$ are the lines going through $v_0$ parallel to the eigenvectors of $N$.

The invariant line associated to the eigenvalue greater than one is stable, while the other invariant line is unstable. To prove that the discrete hyperbola is asymptotic to $L_1$ and $L_2$, we note that the monodromy transformation $M$ sends the discrete hyperbola to itself by mapping $$M\,:\,\phi(\tilde{P}_i)\mapsto \phi(\tilde{P}_{i+n}).$$ Thus, the edges $\phi(\tilde{P}_i)$ of the discrete hyperbola approach the stable and unstable invariant lines of $M$ as the index $i$ approaches positive and negative infinity, respectively. The discrete hyperbola bounds a convex region because its interior angles are less than $\pi$. Any line going though $v_0$ between $L_1$ and $L_2$ eventually intersects this region, because the discrete hyperbola approaches the eigenlines of $M$. Then, convexity implies that the complement of this convex region is star-shaped at $v_0$.

Let $R$ denote the region bounded by $L_1$, $L_2$, and the discrete hyperbola, with partial boundary $\phi(\widetilde{\partial S})$. Let $L$ be any line going through $v_0$ between $L_1$ and $L_2$ (for instance, we may assume $L$ is a line through $v_0$ and a vertex of the discrete hyperbola). Then the region bounded by $L$, $M\cdot L$, and $\phi(\widetilde{\partial S})$ is a fundamental domain for the action of $M$ on $R$, see Figure \ref{fig8}. The orbit space $$C:=\{M^n\,:n\in \Z\}\,\!\backslash R$$ inherits an integral-affine structure from $R$ non-singular away from $v_0$. Furthermore $C$ is cone over $\partial S$ such that $C_{sing}=\{v_0\}$ and the orientations on $C$ and $S$ induce opposite orientations on $\partial S$. Then, we define $\hat{S}$ to be the result of gluing $C$ and $S$ along their boundaries.

\begin{figure}
\centering
\includegraphics[width = 3in]{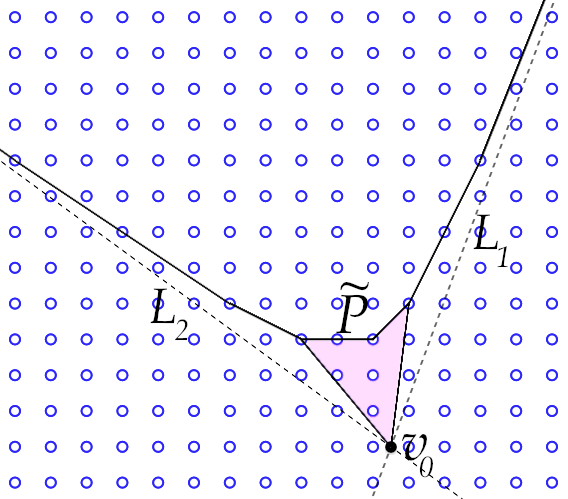}\caption{A fundamental domain for the action of $M$ on $R$.}
\label{fig8}
\end{figure}

More precisely, we can form a quotient $$\{M^n\,:n\in \Z\}\,\!\backslash (\phi(\tilde{U})\cup R)=U\cup C$$ and $U\cup C$ glues to $S$ along $U$ to produce an integral-affine sphere $\hat{S}$ such that $\hat{S}_{sing}=S\cup\{v_0\}$. \end{proof}

We note that $v_0$ may only have rational coordinates. Thus, even if $S$ admits a triangulation, $\hat{S}$ may not. Hence we define:

\begin{definition} Let $S$ be an integral-affine surface. The {\it order $k$ refinement} $S[k]$ is produced by post-composing the charts on $S$ with multiplication by $k$. \end{definition}

Note that $S[k]_{sing}$ and $S_{sing}$ are naturally identified. We remark without proof that the order $k$ refinement of $\Gamma(\mathcal{X}_0)$ corresponds to a SNC resolution of the order $k$ base change of $\mathcal{X}\rightarrow \Delta$.

\begin{remark}\label{refine} Let $S$ be as in Proposition \ref{complete}. Then the point $v_0$ in the definition of $\hat{S}$ may not be integral, but as it is rational, there exists some $k$ such that $v_0$ lies at an integral point of $\hat{S}[k]$. Then if $S$ can be triangulated into basis triangles, so can $\hat{S}[k]=S[k]\cup C[k]$, as $C[k]$ has a fundamental domain given by a lattice polygon, see Figure \ref{fig8}, and thus can easily be triangulated into basis triangles. \end{remark} 

\section{A Proof of Looijenga's Conjecture}

In this section, we present the construction of a Type III anticanonical pair $(\mathcal{X}_0,D)$ from an anticanonical pair $(Y,D)$, thus proving Looijenga's conjecture. But first, we need:

\begin{proposition}\label{model} Every anticanonical pair $(Y,D)$ can be expressed as a sequence of node smoothings and internal blow-ups starting with a toric pair $(\overline{Y},\overline{D})$.\end{proposition}

\begin{proof} First, we express $(Y,D)$ as a sequence of corner and internal blow-ups of a minimal anticanonical pair $(Y_0,D_0)$. We factor the blow-down to $(Y_0,D_0)$ into maps $\alpha$ and $\beta$ $$(Y,D)\xrightarrow{\alpha} (Y_1,D_1)\xrightarrow{\beta} (Y_0,D_0)$$ such that $\alpha$ consists only of interior blow-ups, while $\beta$ consists only of corner blow-ups, cf. Remark 2.6 of \cite{Fri15}. By direct examination (see Lemma 3.2 of \cite{FM83}), every minimal anticanonical pair $(Y_0,D_0)$ is a node smoothing of a minimal toric anticanonical pair, i.e. there is a family of anticanonical pairs with cycle $D_0$ that degenerates to a toric pair. Performing all the corner blow-ups of $\beta$ on this family expresses $(Y_1,D_1)$ as a node smoothing of a toric pair $(\overline{Y},\overline{D})$. Thus, $(Y,D)$ is the result of interior blow-ups and node smoothings on $(\overline{Y},\overline{D})$.\end{proof}

\begin{theorem}[Looijenga's Conjecture] If $D$ is the anticanonical divisor of some rational surface, then the cusp singularity associated to $D'$ is smoothable.\end{theorem}

\begin{proof} Let $(Y,D)$ be an anticanonical pair. We express $(Y,D)$ as a sequence of node smoothings and internal blow-ups on a toric pair $(\overline{Y},\overline{D})$. From this data, we construct a Type III anticanonical pair $\mathcal{X}_0$: \vspace{3pt}

{\bf Construction:} Let $\overline{S}$ be a moment polygon for $(\overline{Y},\overline{D})$ with integral vertices. Then $\overline{S}$ is an integral-affine disc with no singularities, such that $(x_i,y_i)$ is an oriented lattice basis. The components $\overline{P}_i\subset \partial \overline{S}$ of the boundary have negative self-intersections $-\overline{D}_i^2$ for all $i$, in the sense of Definition \ref{selfie}. In fact we may simply choose a polygon $\overline{S}$ with this property, should we wish to avoid an appeal to symplectic geometry.

For each internal blow-up or node smoothing of $(\overline{Y},\overline{D})$ applied to produce $(Y,D)$, we perform an associated internal blow-up or node smoothing surgery on the integral-affine surface $\overline{S}$. Assume that all surgeries have size $1$, i.e. internal blow-ups remove a single basis triangle and node smoothings have cuts of minimal length, as in the example below. We are applying $Q(Y,D)$ surgeries of fixed size, but $\overline{S}$ may be chosen to be arbitrarily large (e.g. by scaling). Thus we can ensure, and in fact we require, that $\overline{S}$ can accommodate all the necessary surgeries. We denote the resulting integral-affine surface by $S$.

By Proposition \ref{same}, the negative self-intersections of the boundary components $P_i\subset \partial S$ are equal to the negative self-intersections $d_i$ of the components of $D$. By Remark \ref{rem}, $S$ is the base of an almost toric fibration of a symplectic rational surface $(Y,D,\omega)$. For example, Figure \ref{fig5} is a moment polygon $\overline{S}$ for the toric surface $$(\mathbb{P}^1\times\mathbb{P}^1,\, \overline{D}_1+\overline{D}_2+\overline{D}_3+\overline{D}_4).$$ Figure \ref{fig6} demonstrates 18 surgeries on $\overline{S}$ corresponding to four internal blow-ups on $\overline{D}_1$, two internal blow-ups on $\overline{D}_2$, six internal blow-ups on $\overline{D}_3$, five internal blow-ups on $\overline{D}_4$ and a node smoothing of $\overline{D}_1\cap \overline{D}_2$. After the surgeries, the integral-affine surface $S$ is the almost toric base of a negative-definite anticanonical pair $(Y,D,\omega)$ with ${\bf d}=(d_1,d_2,d_3)=(4,6,5)$. The charge of $(Y,D)$ is $18$, because each surgery increases the charge by $1$.

\begin{figure}
\includegraphics[width = 3in]{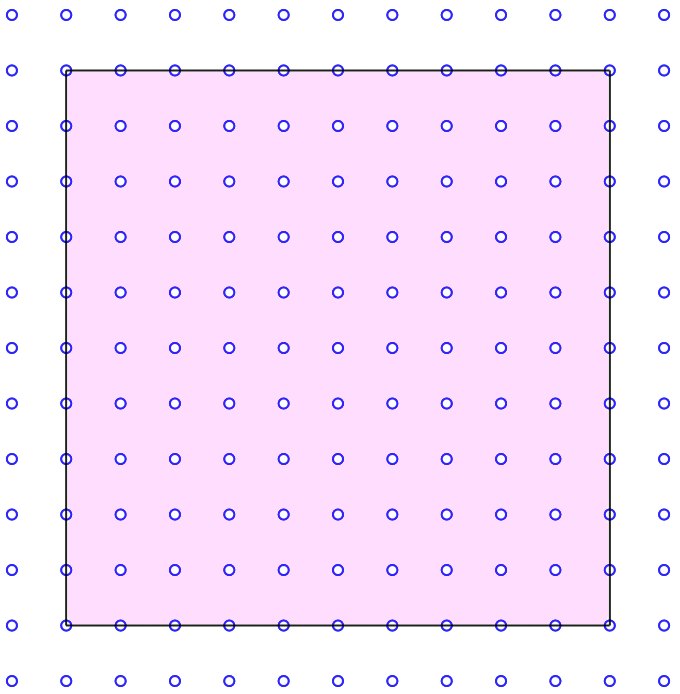}\caption{A moment polygon $\overline{S}$ for $\overline{Y}$.}
\label{fig5}
\end{figure}

\begin{figure}
\includegraphics[width = 3in]{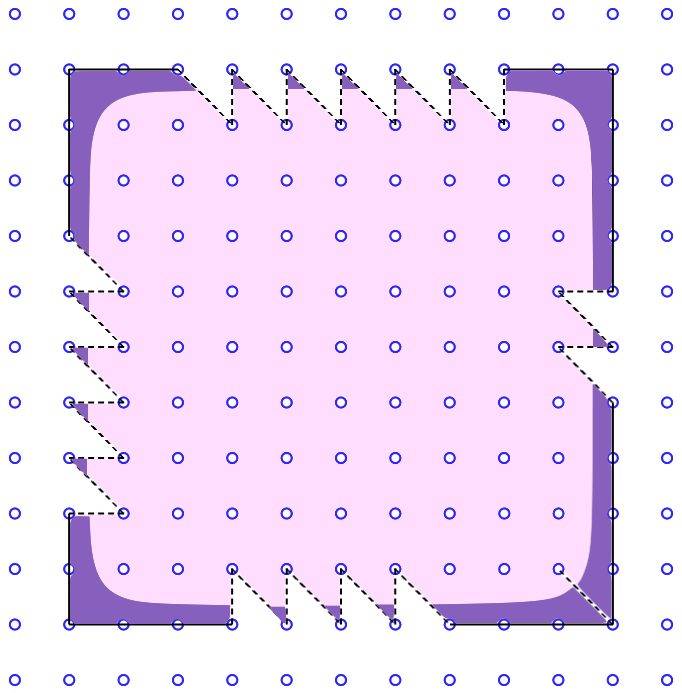} \caption{An almost toric base $S$ for $Y$.}
\label{fig6}
\end{figure}

The surgeries retain the property that $(x_i,y_i)$ is an oriented lattice basis, and hence $\partial S$ satisfies the conditions of Proposition \ref{complete}. Thus, we may complete $S$ to a sphere $\hat{S}=S\cup C$. Note that $S$ admits a triangulation into basis triangles, for instance, we may triangulate the polygonal fundamental domain shown in Figure \ref{fig6}. By Remark \ref{refine}, we may take an order $k$ refinement $\hat{S}[k]$ such that $\hat{S}[k]$ admits a triangulation into basis triangles. Now that we have established the existence of such a triangulation, we choose amongst all of them the one which attains the minimal possible number of edges emanating from $v_0$. 

Let $v_i$ be a vertex of the triangulation of $\hat{S}[k]$ which is non-singular. Then $star(v_i)$ is the pseudo-fan of some toric surface $(V_i,D_i)$. Now suppose that $v_i\in S[k]_{sing}$ is a singular point not equal to $v_0$. Each such singularity $v_i$ is introduced by a surgery on $\overline{S}$. Let $\overline{v}_i\in \overline{S}$ denote the pre-image of $v_i\in S$. In the case of an internal blow-up, one of the triangular faces of $star(\overline{v}_i)$ is collapsed by the surgery, whereas in the case of a node smoothing, the integral-affine structure along an edge of $star(\overline{v_i})$ is changed. In fact:

\begin{enumerate}
\item An internal blow-up on $\overline{S}$ corresponds to a node smoothing on $star(\overline{v}_i)$ by Proposition \ref{node} and Definition \ref{inter}. 
\item A node smoothing on $\overline{S}$ corresponds to an internal blow-up on $star(\overline{v}_i)$ by Proposition \ref{blow} and Definition \ref{smoo}.
\end{enumerate}

\noindent We conclude that there is an anticanonical pair $(V_i,D_i)$ whose pseudo-fan is $star(v_i)$ for all $i\neq 0$.

Finally, we consider $v_0$. The monodromy $N=N(\gamma)$ of a counterclockwise loop around the boundary $\partial S$ is equal to the monodromy of a {\it clockwise} loop around $v_0$. Thus, the monodromy of a counterclockwise loop around $v_0$ is $N^{-1}$.

\begin{lemma} The pseudo-fan of $(V_0,D')$ is isomorphic to $star(v_0)$. \end{lemma}

\begin{proof} Let ${\bf d}_0=(d_{01},\dots,d_{0r})$ denote the negative self-intersections of the edges $(e_{01},\dots,e_{0r})$ emanating from $v_0$. We claim that $d_{0i}\geq 2$ for all $i$. First, we show that $d_{0i}\leq 0$ is impossible. Suppose for the sake of contradiction that $d_{0i}\leq 0$. Then the formula $$d_{0i}e_{0i}=e_{0(i-1)}+e_{0(i+1)}$$ implies that the angle $\angle (v_{i-1}v_0v_{i+1})$ subtended by $star(v_0)$ between $e_{0(i-1)}$ and $e_{0(i+1)}$ is at least $\pi$ in any integral-affine chart. But, by the definition of the integral-affine structure on $C$ from Proposition \ref{complete}, the image of the developing map of $star(v_0)$ lies within the region $R$. Thus $d_{0i}\leq 0$ is impossible, because the image of the developing map of $star(v_0)$ subtends an angle less than $\pi$---it subtends the angle formed at $v_0$ by $L_1$ and $L_2$.

If $d_{0i}=1$, then the union of the two triangles containing $e_{0i}$ is integral-affine equivalent to the unit square. But then we may alter the triangulation by flipping the diagonal of this square, thus decreasing the total number of edges emanating from $v_0$. This contradicts our assumption that the number of edges emanating from $v_0$ is minimal. Hence $d_{0i}\geq 2$ for all $i$. If $d_{0i}=2$ for all $i$, then the image of developing map subtends an angle of exactly $\pi$, which is also impossible. Hence $d_{0i}\geq 3$ for some $i$. Thus ${\bf d}_0$ is negative-definite. We remark that similar ideas arise when constructing minimal resolutions of non-smooth toric surfaces, see Section 2.6 of \cite{Ful93}.

The developing map, when restricted to the boundary of $star(v_0)$, maps to an infinite lattice polygon lying in $R$ and bounded by $L_1$ and $L_2$. Because ${\bf d}_0$ is negative-definite, the angles of this infinite lattice polygon are less than $\pi$, and thus, it bounds a convex region. Furthermore, the image of the developing map of $star(v_0)$ contains no lattice points in its interior because it is a union of basis triangles containing $v_0$. See Figure \ref{newfig}. This uniquely characterizes the image of the developing map of $star(v_0)$: It is the region between $L_1$ and $L_2$ in the complement of the convex hull of the lattice points between $L_1$ and $L_2$. We say $star(v_0)$ has property ($\star$).

\begin{figure}
\centering
\includegraphics[width = 4.5in]{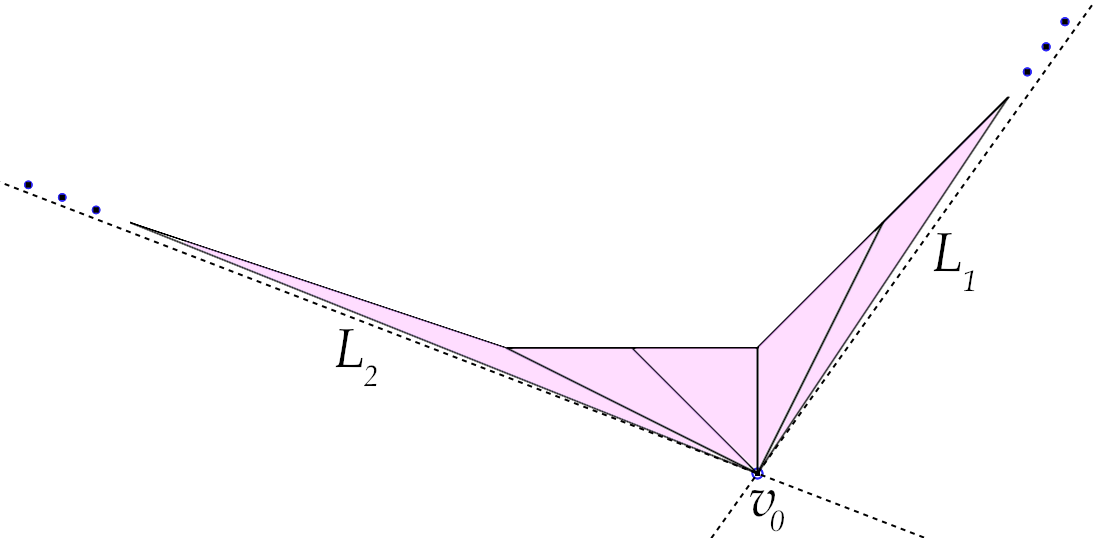}\caption{The image of the developing map of $star(v_0)$.}
\label{newfig}
\end{figure}

Let ${\bf d}=(d_1,\dots,d_n)$ and ${\bf d}'=(d_1',\dots,d_s')$. The following matrices are conjugate in $SL_2(\Z)$: $$ \prod_{i=1}^r \begin{pmatrix} 0 & -1 \\ 1 & d_{0i}\end{pmatrix}\sim N^{-1}=\left[\prod_{i=1}^n\begin{pmatrix} d_i & -1 \\ 1 & 0\end{pmatrix}\right]^{-1}\sim \prod_{i=1}^s \begin{pmatrix} 0 & -1 \\ 1 & d_i'\end{pmatrix}$$ where the last similarity is a general fact about dual cycles, see \cite{Ino77}. That is, the monodromy of $star(v_0)$ is conjugate to the monodromy of the pseudo-fan of $(V_0,D')$. By post-composing with an integral-affine transformation, we may assume that these monodromies are equal and that the developing map of the pseudo-fan of $(V_0,D')$ maps into the region between $L_1$ and $L_2$. Since ${\bf d}'$ is negative-definite, the image of the developing map of the pseudo-fan of $(V_0,D')$ is also characterized by property $(\star)$. Since the monodromy acts the same on these images, the pseudo-fan of $(V_0,D')$ and $star(v_0)$ are isomorphic, as triangulated integral-affine surfaces. \end{proof}

For every vertex $v_i$ with $i\neq 0$ of the triangulation of $\hat{S}[k]$, we have found an anticanonical pair $(V_i,D_i)$ whose pseudo-fan is $star(v_i)$. In addition, we have proved that the pseudo-fan of $(V_0,D')$ is $star(v_0)$. Consider the union of the surfaces $$\mathcal{X}_0:=\bigcup_{v_i\in \hat{S}[k]} (V_i,D_i)$$ where we identify $D_{ij}$ with $D_{ji}$ so that nodes of $D_i$ are identified with nodes of $D_j$. By Remark \ref{tpf}, $\mathcal{X}_0$ satisfies the triple point formula. So $\mathcal{X}_0$ satisfies all the assumptions of a Type III anticanonical pair $(\mathcal{X}_0,D)$. Theorem \ref{fm} implies that the cusp with resolution $D'$ is smoothable. \end{proof}

We now describe without proof a number of modifications of the construction, which still produce a Type III anticanonical pair $\mathcal{X}_0$ but in which various conditions are weakened.

\begin{modification}\label{overlap} We need not assume that the singularities introduced in the surgeries on $\overline{S}$ are distinct. The number of surgeries in which the vertex $v_i$ is involved (either as the vertex of a triangle removed from $\overline{S}$ for an internal blow-up, or as the end of a cut for a node smoothing) is equal to the charge $Q(V_i,D_i)$ of the anticanonical pair whose pseudo-fan is $star(v_i)$. In addition, the edges of the triangles removed for internal blow-ups may overlap, and may also overlap cuts for node smoothings. \end{modification}

\begin{modification}\label{less} We assumed that for every component of $\overline{D}$, the length of the associated boundary component of $\overline{S}$ was positive. This assumption is unnecessary---some may have length zero (but at least three edges must have positive length, because $\overline{S}$ must have nonempty interior to make surgeries). \end{modification}

\begin{modification} The internal blow-ups on a boundary component of $\overline{S}$ decrease its length. After the surgeries, we may allow some of the boundary components of $S$ to have length zero. When only some of the boundary components have length zero, we continue with the construction by applying the developing map to a collar neighborhood of the boundary, even though $(x_i,y_i)$ may cease to be a lattice basis.\end{modification}

\begin{modification}\label{empty} If, after the surgeries, {\it all} the boundary components have length zero, then $S$ has no boundary, and is already homeomorphic to a sphere. Then we may triangulate $S$ directly to produce the dual complex $\Gamma(\mathcal{X}_0)$ of a Type III anticanonical pair.\end{modification}

\begin{remark} Modification \ref{empty} is always possible, thus eliminating the need for the completion process described in Proposition \ref{complete}, but the construction is significantly more delicate.
\end{remark}

We now present an example incorporating Modifications \ref{overlap}, \ref{less}, and \ref{empty} of the construction of $\mathcal{X}_0$:

\begin{example} Let $(Y,D)$ be a negative-definite anticanonical pair such that $Q(Y,D)=3$. It can be shown that all negative-definite anticanonical pairs with $Q(Y,D)=3$ have three disjoint internal exceptional curves, which can be blown down to a toric surface $$\pi\,:\,(Y,D)\rightarrow(\overline{Y},\overline{D}).$$ We call $(\overline{Y},\overline{D})$ a {\it toric model}. Let $\overline{D}_{a_1}$, $\overline{D}_{a_2}$, and $\overline{D}_{a_3}$ be the three components of $\overline{D}$ that receive the internal blow-ups. Up to scaling by $\Z$, there is a unique moment polygon $\overline{S}$ of the toric model in which the boundary components $\overline{P}_{a_1}$, $\overline{P}_{a_2}$, and $\overline{P}_{a_3}$ are the only edges with nonzero length.

To construct $S$, we perform internal blow-ups on the three boundary components of $\overline{S}$. To do so, we must delete a multiple of a basis triangle from each of the edges. For any anticanonical pair $(Y,D)$ of charge three, it can be proven that $\overline{S}$ has room to perform internal blow-ups that decrease the lengths of all three edges to zero. So $\partial S$ is empty, and $S$ may be directly triangulated into basis triangles, from which we construct $\mathcal{X}_0$.

\begin{figure}
\includegraphics[width = 3.2in]{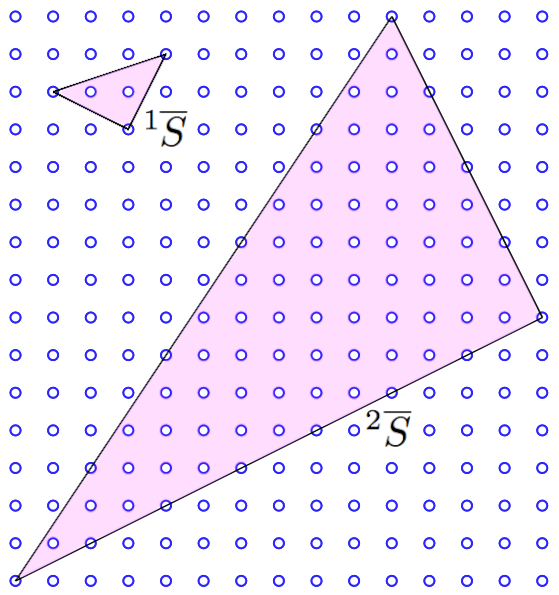}\caption{Moment polygons for $(^i\overline{Y},\,^i\overline{D})$.}
\label{fig9} \vspace{10pt}
\includegraphics[width = 3.2in]{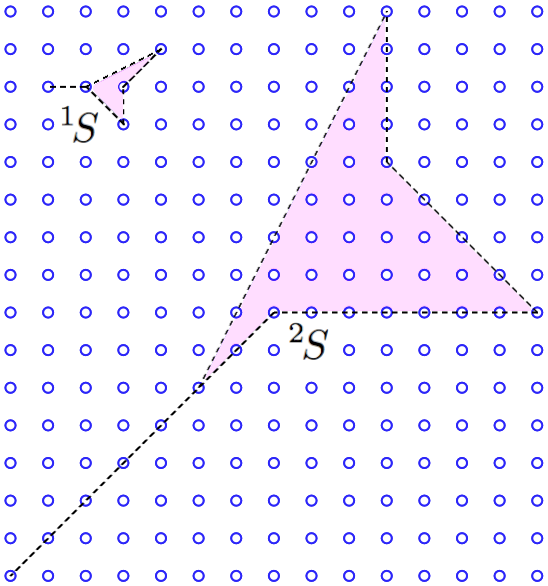} \caption{Almost toric bases for $(^iY,D)$.}
\label{fig10}
\end{figure}

Consider the cusp singularity $D'$ with cycle ${\bf d}'=(6,9)$. The dual cycle is given by the formula $${\bf d}=(3,2,2,2,3,2,2,2,2,2,2).$$ There are two distinct deformation families of pairs with anticanonical cycle $D$. The deformations preserve the classes of exceptional curves, so each deformation family is associated to a different toric model. Let $(^iY,D)$ with $i=1,2$ be two anticanonical pairs representing these two deformation families. The cycles of negative self-intersections of the two toric models are \begin{align*} ^1{\bf \overline{d}}&=(3,2,1,2,3,1,2,2,2,2,1) \\ ^2{\bf \overline{d}}&=(3,2,2,1,3,2,1,2,2,2,1) \end{align*}

By blowing down exceptional curves in $^i\overline{D}$, we can draw fans for $(^i\overline{Y},\,^i\overline{D})$, from which we can construct moment polygons $^i\overline{S}$ for $i=1,2$. Using Modification \ref{less}, we choose a moment polygon $^1\partial\overline{S}$ whose only boundary components of positive length are $^1\overline{P}_3$, $^1\overline{P}_6$, and $^1\overline{P}_{11}$, while the only components of $^2\partial \overline{S}$ of positive length are $^2\overline{P}_4$, $^2\overline{P}_7$, and $^2\overline{P}_{11}$ as in Figure \ref{fig9}.

\begin{figure}
\includegraphics[width = 4.5in]{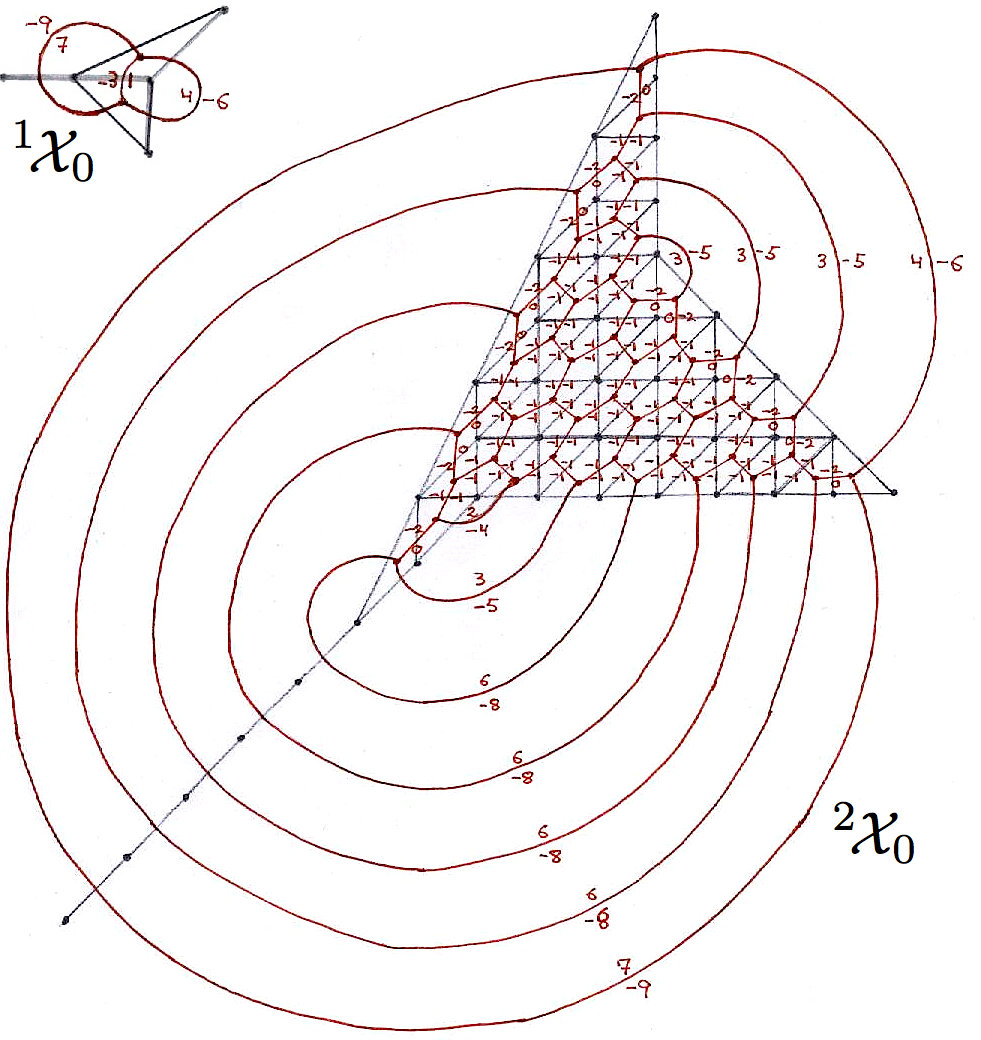}\caption{Two Type III anticanonical pairs $(^i\!\mathcal{X}_0,D)$.}
\label{fig11}
\end{figure}

We perform three internal blow-ups on $^i\overline{S}$ by deleting a multiple of a basis triangle resting on each of the three edges, then gluing the remaining two edges of each triangle. Furthermore, using Modification \ref{empty}, we choose surgeries large enough to reduce the length of the boundary to zero. The resulting integral affine surfaces  $^i\!S$ are shown Figure \ref{fig10}. Because $^i\!S$ has no boundary, we immediately triangulate it into basis triangles, and construct the simple normal crossings surface $^i\!\mathcal{X}_0$ whose dual complex is $^i\!S$. Note that Modification \ref{overlap} has been applied to $^1S$ as two singularities introduced by surgeries overlap. The triangulations of $^i\!S$ and the surfaces $^i\!\mathcal{X}_0$ are shown in Figure \ref{fig11}. The hyperbolic Inoue pair $(V_0,D')$ is the outer face in both illustrations.

The surface $^i\!\mathcal{X}_0$ smooths to give a family $^i\!\mathcal{X}\rightarrow \Delta$ of surfaces over the disc whose general fiber is a pair with anticanonical cycle $D$. It is a natural question to ask whether the general fiber of $^i\!\mathcal{X}$ is deformation-equivalent to $(^iY,D)$. In later work, we provide an affirmative answer to this question, verifying Conjecture 6.1 of \cite{FM83}. To prove smoothability of the cusp $D'=(6,9)$, Friedman and Miranda exhibited the special fiber $^1\!\mathcal{X}_0$. The surface $^2\!\mathcal{X}_0$ is the alternative special fiber with $50$ triple points that they conjectured to exist. 


\end{example}

\end{document}